\documentclass[11pt]{article}

\usepackage{amssymb}
\usepackage{amsthm}
\usepackage{ifthen,amsfonts,graphicx,latexsym,algorithm,algorithmic,url}
\usepackage{amsmath,color}
\usepackage{verbatim}
\usepackage[latin1]{inputenc}
\usepackage[english]{babel}
\usepackage{esint}
\usepackage{bbm}
\usepackage{listings}
\usepackage{color}
\usepackage{graphicx, wrapfig}
\usepackage[margin=1in,footskip=0.25in]{geometry}
\newtheorem{theorem}{Theorem}[section]
\newtheorem{lemma}[theorem]{Lemma}
\newtheorem{proposition}[theorem]{Proposition}
\newtheorem{corollary}[theorem]{Corollary}

\newtheorem{defi}[theorem]{Definition}
\theoremstyle{remark}
\newtheorem{remark}{Remark}

\DeclareMathOperator{\im}{Im}

\DeclareMathOperator{\identity}{Id}
\DeclareMathOperator{\rank}{rank}
\DeclareMathOperator{\prob}{Prob}
\DeclareMathOperator{\dx}{dx}
\DeclareMathOperator{\dy}{dy}
\DeclareMathOperator{\curl}{curl}
\DeclareMathOperator{\dnux}{d\nu_x}
\DeclareMathOperator{\dnu}{d\nu}
\DeclareMathOperator{\span2}{span}

\def\ccrn{\mathbb{R}^n}
\def\ccr{\mathbb{R}}
\def\ccrd{\mathbb{R}^n}

\def\ccrN{\mathbb{R}^N}

\def\grad{\nabla}
\def\calA{\mathcal{A}}
\def\calM{\mathcal{M}}
\def\bbA{\mathbb{A}}
\def\bbP{\mathbb{P}}

\def\bbQ{\mathbb{Q}}
\def\linear{\mathcal{L}}
\def\torusN{\mathrm{T}_N}

\def\weakConv{\rightharpoonup}

\def\id{\identity}

\def\overlineF{\overline{F}}

\def\grad{\nabla}

\def\hpo{\mathbb{H}^p_0}
\def\ccrNn{\mathbb{R}^{d}}
\def\ccrnn{\mathbb{R}^{d}}
\def\intccrnn{\int_{\ccrnn}}
\def\weaksconv{\overset{\ast}{\rightharpoonup}}
\def\translation{\delta}
\def\nuxi{\nu_{\xi}}
\def\calF{\mathcal{F}}

\def\torusn{\mathrm{T}^{N}}
\def\inttorusn{\int_{\torusn}}
\def\indyk{\mathbbm{1}}
\def\naturals{\mathbb{N}}
\def\qa{Q_{\calA}}

\def\weakstarconv{\overset{\ast}{\rightharpoonup}}
\def\intccrd{\int_{\ccrd}}
\def\wntilde{\widetilde{W}_M}
\def\gm{\mathcal{G}_M}
\def\amp{A^{+}}
\def\ccrmn{\mathbb{R}^{m \times n}}
\def\ccrmr{\mathbb{R}^{m \times r}}

\def\ccrrn{\mathbb{R}^{r \times n}}
\def\ccrm{\mathbb{R}^{m}}

\def\weakStarConv{\overset{*}{\rightharpoonup}}
\def\intOmega{\int_{\Omega}}
\def\W1p{\mathrm{W}^{1,p}}
\def\Wm1p{\mathrm{W}^{-1,p}}
\def\LL{\mathrm{L}}
\def\WW{\mathrm{W}}
\def\II{\mathrm{I}}

\def\og{\leavevmode\raise.3ex\hbox{$\scriptscriptstyle\langle\!\langle$~}}
\def\fg{\leavevmode\raise.3ex\hbox{~$\!\scriptscriptstyle\,\rangle\!\rangle$}}

\vspace{9mm}
\title{Closed $\calA$-$p$ Quasiconvexity and Variational Problems with Extended Real-Valued Integrands}
\vspace{9mm}
\author{Adam Prosinski \thanks{University of Oxford, EPSRC CDT in Partial Differential Equations {\tt adam.prosinski@maths.ox.ac.uk} }}
\vspace{9mm}

\begin{document}
\maketitle

\abstract{This paper relates the lower semi-continuity of an integral functional in the compensated compactness setting of vector fields satisfying a constant-rank first-order differential constraint, to closed $\calA$-$p$ quasiconvexity of the integrand. The lower semi-continuous envelope of relaxation is identified for continuous, but potentially extended real-valued integrands. We discuss the continuity assumption and show that when it is dropped our notion of quasiconvexity is still equivalent to lower semi-continuity of the integrand under an additional assumption on the characteristic cone of $\calA$.}

\section{Introduction}
A classical problem in the Calculus of Variations is determining sequential lower semi-continuity criteria for functionals of the form $V \mapsto \intOmega F(V(x)) \dx$ and investigating the different notions of convexity that arise. The full statement of such a problem involves specifying the assumptions, that are essentially of two kinds. One needs information on the integrand $F$ - its regularity and growth rate. Secondly, the class of admissible vector fields $V$ needs to be defined, together with the relevant notion of convergence $V_j \rightarrow V$ under which the sequential lower semi-continuity is to be investigated. 
A typical setting is when one imposes $p$-growth conditions (from above and/or below) on a lower semi-continuous integrand $F$ and considers weak $\LL^p$ convergence of vector fields $V_j$. To obtain less constrictive conditions on the integrand one often restricts the type of test vector fields considered and passes to the so-called compensated compactness setting.

Extensive research has been done in the setting of gradients of Sobolev functions, that is when $V = \grad u$ for some $u \in \W1p$ and the sequential lower semi-continuity is tested on $\grad u_j \weakConv \grad u$ weakly in $\LL^p$. Then the key condition is quasiconvexity and its variants -  there is an abundance of results in this framework available in the literature. We do not attempt to give a comprehensive list of such, and instead we refer the reader to \cite{Dacorogna07} and the bibliography therein for a good overview of the theory. However we feel obliged to give at least an example of one of the classical results, in this generality, due to Acerbi and Fusco. In \cite{AcerbiFusco84} they have shown that if $f \colon \ccrn \times \ccrm \times \ccr^{n \times m} \to \ccr$ is a Carath\'{e}odory integrand satisfying, for some $p \geqslant 1$,
$$ 0 \leqslant f(x,s,\xi) \leqslant a(x) + C(|s|^p + |\xi|^p) \text{ for every } x \in \ccrn, s \in \ccrm, \xi \in \ccr^{n \times m},$$
with some non-negative constant $C$ and a non-negative, locally integrable function $a$, then for any open set $\Omega \subset \ccrn$ the functional $u \mapsto \intOmega f(x, u, \grad u) \dx$ is sequentially weakly (weakly* if $p = \infty$) lower semi-continuous on $\W1p(\Omega; \ccrm)$ if and only if $f$ is quasiconvex in $\xi$. 
It is worth pointing out that this result has been improved upon shortly after by Marcellini in \cite{Marcellini85}. There the author allows for slightly more general growth conditions and presents an alternative approach - for details we refer the reader to the original paper. In the subsequent years there has been a number of other improvements, and one that is particularly relevant for the present work is \cite{Kristensen15}, where the upper growth bounds are dropped and the integrand is allowed to take the value $+ \infty$. Finally, let us remark that quasiconvexity is normally considered in the multi-dimensional case $n, m > 1$, i.e. when the gradient is a matrix. If either the source's or the target's dimension is equal to $1$, quasiconvexity reduces to standard convexity (see for example Theorem 1.7 in \cite{Dacorogna07}), thus making this a very special case. 

To show the similarity between the gradient case and the one studied here observe that requiring that all the vector fields $V$, $V_j$ be gradients is equivalent to requiring that they be zeros of the differential operator $\curl$ (for a matrix this means that every row is $\curl$-free). Here instead of $\curl$ we consider a general first-order constant-rank differential operator $\calA$ (see the next section for precise definitions) and we require that $\calA V_j \rightarrow \calA V$ strongly in $\Wm1p(\Omega)$, which is a natural relaxation of $\calA V_j = \calA V = 0$.
The setting with a general differential operator $\calA$ has been previously studied in the literature,
foundations for it were developed by a number of authors, including the works of Dacorogna (see \cite{Dacorogna82}), Murat (see \cite{Murat78}) and Tartar (see \cite{Tartar79}) to name a few. A paper that is particularly relevant to the present work is \cite{FonsecaMuller99} by Fonseca and M\"{u}ller. In fact many of the preliminary results on the structure of $\calA$-free vector fields that we use here come from that paper. The main point of the present work is to  remove the upper growth bounds on the integrand considered, in particular, to allow $F$ to take the value $+ \infty$.

Besides \cite{FonsecaMuller99} there are two other papers that must be quoted here. First is \cite{FonsecaLeoniMuller04} which relaxes the growth conditions to non-standard $(p,q)$ ones with $p < q$ in the spirit of \cite{FonsecaMaly97} ($p > 1$) and \cite{Kristensen98} ($p = 1$) in the gradient case. It is interesting to note at this point that due to the presence of a gap between the upper growth bound put on the integrand ($q$-growth) and the class of admissible test functions ($\LL^p$) there is a certain choice to be made when defining the functional $V \mapsto \intOmega f(V(x)) \dx$. One may simply consider the pointwise composition $f(V(x))$ and then deal with the fact that the integral in question need not be finite for a general $V \in \LL^p$. Second possible route is to adopt a Lebesgue-Serrin type definition, that is approximate (in the $\LL^p$ sense) the vector field $V$ by smooth vector fields $V_n$ of appropriate growth (that is in $\LL^q$) and consider $\inf \left\{ \liminf \intOmega f(V_n(x)) \dx \right\}$. Naturally, the limits may still turn out to be infinite, but now each integral $\intOmega f(V_n(x)) \dx$ is well defined and finite. This is the approach used in \cite{FonsecaLeoniMuller04} and we refer as well to \cite{Marcellini86}, where its validity is discussed in the gradient case. 

Another important improvement of \cite{FonsecaMuller99} particularly relevant to the present work was made in \cite{BraidesFonsecaLeoni00}. While \cite{FonsecaMuller99} forms the foundation of the study and identifies $\calA$-quasiconvexity as an equivalent condition for sequential lower semi-continuity of integral functionals given by a continuous integrand $f$ satisfying upper $p$-growth bounds, \cite{BraidesFonsecaLeoni00} studies a relaxation of this problem in the appropriate sense. It is shown that under the same continuity and growth assumptions on the integrand as in \cite{FonsecaMuller99} it is possible to identify the sequential lower semi-continuous envelope (where the notion of convergence is $V_j \weakConv V$ in $\LL^p$ and $\calA V_j \to \calA V$ in $\Wm1p$) of the functional $V \mapsto \intOmega f(V(x)) \dx$. The authors prove that this relaxed problem has an integral representation, with the integrand being the $\calA$-quasiconvex envelope of $f$. The main result of the present work is in the same spirit, but we wish to drop the upper growth-bound and allow the integrand to take the value $+ \infty$. 
Let us note that the results we cite are in fact more general than what we discuss here in the sense that the integrand $f$ is allowed to depend on more variables, but in the brief outline above we decided to opt for simplicity, to emphasize the main features of the respective contributions. 

Passing from the gradient case to a general operator $\calA$ is of interest because of the scope of applications. A number of examples of operators satisfying the constant rank condition may be found in Section 3 of \cite{FonsecaMuller99}. Those include the $\curl$ operator which, as mentioned before, corresponds to the case of vector fields which are gradients. This has been extensively studied on its own and large part of research on constant rank operators aims to reproduce, in this more general setting, the results already available for gradients.
It is interesting to note that one may also study the case of symmetrised gradients (of interest in the theory of elasticity) or gradients (derivatives) of order higher than one in the $\calA$-free framework, as pointed out in Example 3.10 of the aforementioned paper. 
Finally, the requirement of being divergence-free may also be phrased in the language of constant rank operators. Moreover, a mixture of $\text{div}$-free and $\text{curl}$-free conditions relating the magnetisation and the induced magnetic field may be expressed through a suitable constant rank differential operator. Thus, the theory of $\calA$-quasiconvexity may also be related to micromagnetics, as pointed out in \cite{FonsecaMuller99} and \cite{FonsecaKruzik10} (see also the references therein, for example \cite{DeSimone93}).

To finish the discussion of different operators $\calA$ considered in the literature we remark that, while all previous examples correspond to an operator with constant coefficients (i.e. independent of $ x \in \Omega$), there has also been some work on the case of $\calA(x)$ varying with $x$. An example of such a result is given in \cite{Santos04}, where the author generalises the lower semi-continuity results of \cite{FonsecaMuller99} to the case $\calA(x)$. Note that the constant rank hypothesis is still in place, and the rank must not depend on $x$.

Let us also mention that while this work focuses on studying oscillation phenomena in $\calA$-free sequences of functions, it is also possible to include concentration effects. In this case one switches from the classical Young measures we use here to the so-called generalised Young measures (see \cite{DiPernaMajda87}). Some recent results in this matter may be found, for example, in \cite{FonsecaKruzik10} or \cite{ArroyoPhilippisRindler17}.

\subsection{Announcements of results}
Throughout the paper the standard assumption is that the first order differential operator $\calA$ satisfies the constant-rank condition. As a mixture of notions present in \cite{FonsecaMuller99} and \cite{Kristensen15} (see also \cite{Pedregal97}) we say that a function is \textit{closed} $\calA$-$p$ quasiconvex if it satisfies Jensen's inequality with respect to all homogeneous Young measures generated by $\LL^p$-weakly convergent $\calA$-free vector fields. 

Our first result, Theorem \ref{thmQCimpliesLSC}, shows that this notion of quasiconvexity is, for a non-negative integrand, sufficient for lower semi-continuity of the functional $\II_F[V] := \intOmega F(V(x)) \dx$ in the sense outlined before.   

Under additional lower-growth bound of the type $F(\xi) \geqslant |\xi|^p$ and a continuity assumption we obtain, in Theorem \ref{thmContinuousImpliesRelaxation}, a full characterisation of the lower semi-continuous envelope of the functional $\II_F$. In this case we show that the relaxed problem is given by integration of the closed $\calA$-$p$ quasiconvex envelope of $F$. 

Finally we remove the continuity assumption and replace it with the requirement that $F$ be real valued in Theorem \ref{thmLSCimpliesQC}. There we show that if the characteristic cone of the operator $\calA$ spans the entire space then closed $\calA$-$p$ quasiconvexity is still equivalent to lower semi-continuity of the functional $\II_F$. 

The paper is organised as follows. In Section 2 we introduce the principal notions and fundamental results. Notably we define the class of $\calA$-$p$ Young measures and give regularisation results for sequences generating such measures. The last part of this Section characterises the measures in question in terms of $\calA$-$p$ quasiconvex functions. The contents of Section 2 are mostly technical results borrowed from \cite{FonsecaMuller99}, to which we refer for proofs. 

The main part of the paper is contained in Section 3, where we introduce the notion of closed $\calA$-$p$ quasiconvexity. The first thing we prove is its sufficiency for lower semi-continuity of the functional. The main ingredient for the proof of necessity results is the representation formula for the closed $\calA$-$p$ quasiconvex envelope given in Proposition \ref{lemmaQCenvelope}. The proof of it is quite complicated, relies on the Kuratowski Ryll-Nardzewski Measurable Selection Theorem (see Theorem \ref{thmKuratowski}), and encompasses many of the difficulties encountered throughout the proofs of the other main results. In this context this particular strategy is, to the author's knowledge, different from that used in other proofs of similar results, although it is worth pointing out that Sychev (see \cite{Sychev99}) also uses a measurable selection argument to obtain a representation of a quasiconvex envelope. Once this representation formula is established we largely follow the spirit of Kristensen's (see \cite{Kristensen15}) proof for corresponding results in the case of gradients.

The Appendix contains a proof of an auxiliary result necessary for dealing with sequences of $\calA$-free vector fields and Young measures generated by them. This has been used in previous papers on $\calA$-quasiconvexity, but the approach we show here is, to the author's knowledge, new and more elementary than previous arguments. 

\subsection{Acknowledgements}
I would like to thank my advisor Prof. Jan Kristensen for his valuable advice and guidance throughout this project. I also thank the Oxford EPSRC CDT in Partial Differential Equations, the Clarendon Fund, and St John's College Oxford, whose generous support is gratefully acknowledged.

\section{Notation and preliminary results}
We begin by introducing the language of Young measures and the Fundamental Theorem of Young Measures (Theorem \ref{thmFToYM}) in particular. These results are classical and a typical reference is \cite{Pedregal97}, where the relevant proofs may be found. 
Here and in all that follows $\Omega \subset \ccrN$ is an open and bounded domain with $|\partial \Omega| = 0$, where $\partial \Omega$ is the boundary and $|\cdot|$ denotes the $N$-dimensional Lebesgue measure. We write $\calM (\ccrd)$ for the space of finite Radon measures on $\ccrd$. 

\begin{theorem}\label{thmFToYM}
Let $E \subset \ccrN$ be a measurable set of finite measure and $z_j \colon E \rightarrow \ccrn$ be a bounded sequence of $\LL^p$ functions for some $p \in [1, \infty]$. Then there exists a subsequence $z_{j_k}$ and a weak$^*$-measurable map $\nu \colon E \rightarrow \calM (\ccrn)$ such that the following hold:

i) every $\nu_x$ is a probability measure,

ii) if $f \colon \Omega \times \ccrn \rightarrow \ccr \cup \{\infty\}$ is a normal integrand bounded from below, then
$$ \liminf_{j \rightarrow \infty} \int_{\Omega} f(x, z_{j_k}(x)) \dx  \geqslant \int_{\Omega} \overline{f}(x)\dx,$$
where
$$ \overline{f}(x) := \langle \nu_x, f(x,\cdot) \rangle = \int_{\ccrn}f(x,y) \dnux(y);$$

iii) if $f \colon \Omega \times \ccrn \rightarrow \ccr \cup \{\infty\}$ is Carath\'{e}odory and bounded from below, then
$$ \lim_{j \rightarrow \infty} \int_{\Omega} f(x, z_{j_k}(x)) \dx  = \int_{\Omega} \overline{f}(x)\dx < \infty$$
if and only if $\{f(\cdot, z_{j_k}(\cdot))\}$ is equiintegrable (in the usual, $\LL^1$ sense). In this case
$$ f(\cdot, z_{j_k}(\cdot)) \weakConv \overline{f} \text{ in } \LL^1(\Omega) \cdot$$
The family $\{ \nu_x \}_{x \in E}$ is called the Young measure generated by $z_{j_k}$. If there exists some $x_0 \in E$ such that $\nu_x = \nu_{x_0}$ for almost every $x \in E$ then we say that $\nu$ is a homogeneous Young measure and often identify the family $\{\nu_x\}$ with the single measure $\nu_{x_0}$ if there is no risk of confusion.
\end{theorem}

Recall that we say that a function $f \colon \Omega \times \ccrd \rightarrow (-\infty, \infty]$ is a normal integrand if $f$ is Borel measurable and for every $x \in \Omega$ the function $z \mapsto f(x, z)$ is lower semi-continuous. Similarly we say that a function $f \colon \Omega \times \ccrd \rightarrow \ccr$ is Carath\'{e}odory if both $f$ and $-f$ are normal integrands. Finally, a map $\nu \colon E \rightarrow \calM (\ccrd)$ is said to be weak*-measurable if $x \mapsto \langle \nu_x, \varphi \rangle$ is measurable for any continuous and compactly supported function $\varphi \colon \ccrd \to \ccr$.

\begin{proposition}\label{propTranslatedYM}
If $\{v_j\}$ generates a Young measure $\nu$ and if $w_j \rightarrow w$ in measure, then $\{v_j + w_j\}$ generates the translated Young measure
$$\widetilde{\nu}_x := \delta_{w(x)} \ast \nu_x,$$
where
$$\langle \delta_{a} \ast \mu, \varphi \rangle = \langle \mu, \varphi(\cdot + a) \rangle$$
for $a \in \ccrd$ and $\varphi \in C_0(\ccrd)$. In particular, if $w_j \rightarrow 0$ in measure, then $\{v_j + w_j\}$ still generates $\nu$. Similarly, if $\| v_j - w_j \|_p \rightarrow 0$ for some $p \in [1,\infty]$ then both $v_j$ and $w_j$ generate the same Young measure. 
\end{proposition}

The following is a classical result and its proof may be found, for example, in the Appendix of \cite{BallMurat84}. Here and in what follows $\torusN$ stands for the $N$-dimensional torus.
\begin{lemma}\label{lemmaOscillationConv}
Let $w \in \LL^p(\torusN; \ccrd)$ with $1 \leqslant p \leqslant \infty$, and set $w_j(x) := w(jx), n \in \mathbb{N}$. Then for any bounded open set $E \subset \ccrN$ we have
$$ w_j \rightharpoonup \int_{\torusN} w(y)\dx \quad \text{in } \LL^p(E; \ccrd) \quad (\overset{*}{\rightharpoonup} if p = \infty) \ldotp$$

In particular the sequence $\{w_j\}$ generates the homogeneous Young measure $\nu := \overline{\delta_w}$, where
$$ \langle \overline{\delta_w}, \varphi \rangle := \int_{\torusN} \varphi(w(y)) \dx \quad \text{ for all } \varphi \in C_0(\ccrd) \ldotp$$
\end{lemma}

Before we move on to Young measures associated with a differential operator let us precise the conditions we put on the operator. Let $\{A^i\}_{i \in \{1, \ldots N\}}$ $\subset \linear(\ccrd; \ccr^d)$ be a collection of linear operators. Define
$$ \calA := \sum_1^N A^i \frac{\partial v}{\partial x_i} \quad \text{for } v \colon \ccrN \rightarrow \ccrd ,$$
and
$$ \bbA(w) := \sum_1^N w_i A^i \quad \text{for } w \in \ccrN\ldotp$$
\begin{defi}
We say that $\bbA$ satisfies the constant rank property if there exists an $r \in \naturals$ such that 
$$ \rank(\bbA(w)) = r \quad \text{for all } w \in S^{N-1} \ldotp$$
\end{defi}
The constant rank property of $\bbA$ is a standing assumption throughout the paper. It is both classical and essential for our study and without it not much is known. Its principal purpose is to allow for an analogue of Helmholtz decomposition in the case of gradients - more on that may be found in the Appendix. Note that this decomposition is crucial for the regularisation results to come next. They are taken from the paper by Fonseca and M\"{u}ller (see \cite{FonsecaMuller99}) and correspond to Lemma 2.15 and Proposition 3.8 therein.

Before we proceed to state those results let us remark that the constant rank assumption originated in the work of Murat (see \cite{Murat81}) and the majority of results are only available when it is satisfied. Nevertheless, there has been some work on non-constant rank operators $\calA$ - we refer the interested reader to \cite{Muller99} for an example of a result in this framework, but in the present work we content ourselves with the usual assumptions regarding $\calA$ and focus on relaxing the conditions imposed on the integrand.

\begin{lemma}\label{lemmaEquintegrableGenerators}
Fix $1 < p < \infty$ and let $\{u_j\}$ be a bounded sequence in $\LL^p(\Omega; \ccrd)$ such that $\calA u_j \rightarrow 0$ in $\Wm1p(\Omega)$ and $u_j \weakConv u$ in $\LL^p(\Omega;\ccrd)$. Assume that $\{u_j\}$ generates the Young measure $\nu$. Then there exists a $p$-equiintegrable sequence $\{v_j\} \subset \LL^p(\Omega; \ccrd) \cap \ker \calA$ such that
$$ \int_{\Omega} v_j \dx = \int_{\Omega} u_j \dx  \quad \text{and} \quad \left|\left|v_j - u_j\right|\right|_{\LL^q(\Omega)} \rightarrow 0 \quad \text{for all } 1 \leqslant q < p \ldotp$$
In particular, $v_j$ still generates $\nu$. 
\end{lemma}

\begin{proposition}\label{propLocalisation}
Let $1 < p < \infty$ and let the family $\{v_j\} \subset \LL^p(\Omega; \ccrd)$ be $p$-equiintegrable. Assume that
$$ \calA v_j \rightarrow 0 \quad \text{ in } \Wm1p,$$
that $\{v_j\}$ generates the Young measure $\nu = \{\nu_a\}_{a \in \Omega}$ and that $v_j \weakConv v$ in $\LL^p$. 
Then for almost every $a \in \Omega$ there exists a $p$-equiintegrable family $\{ u_j \} \subset \LL^p(\torusN; \ccrd) \cap \ker \calA$ generating the homogeneous Young measure $\nu_a$ and satisfying 
$$ \int_{\torusN} u_j \dx = \langle \nu_a, \id \rangle = v(a) \ldotp$$
\end{proposition}

The final technical ingredient of this work is the characterisation of the structure of the space of Young measures generated by $\calA$-free sequences. Similarly to the previous results this has been developed in \cite{FonsecaMuller99} and we refer to this paper for the proofs, notably for our Propositions \ref{propCharacterizationOfYM} and \ref{propCharacterisationYMnonhom}, which correspond to Propositions 4.3 and 4.4 in the cited paper. 

Let 
$$ E := \left\{ g \in C(\ccrd) \colon \lim_{\left|z\right| \rightarrow \infty} \frac{g(z)}{1 + \left|z\right|^p} \text{ exists in } \ccr \right\},$$
be equipped with the norm
$$ \left|\left|g\right|\right|_E := \sup_{z \in \ccrd} \frac{\left|g(z)\right|}{1 + \left|z\right|^p} \ldotp$$
This space is canonically isomorphic to the space of continuous functions on $C(S^{d})$ ($S^{d}$ is seen as the one-point compactification of $\ccrd$) equipped with the sup-norm. In particular it is a separable Banach space and its dual $E^*$ may be identified with the space of Radon measures on $\ccrd \cup \{ \infty \}$.
Therefore if $\nu$ is a probability measure on $\ccrd$ with finite $p$-th moment then $\nu \in E^*$ as for all $g \in E$ we may estimate
$$ \left| \intccrd g \dnux \right| \leqslant \left|\left|g\right|\right|_E \intccrd (1 + \left|z\right|^p) \dnux(z) \ldotp$$
Furthermore we immediately see by taking $g(z) := 1 + \left|z\right|^p$ that 
$$ \left|\left| \nu \right|\right|_{E^*} = \int_{\ccrd} 1 + \left|z\right|^p \dnux(z) \ldotp$$

In this subsection we aim to use the structure of the dual space $E^{\ast}$ to investigate the properties of Young measures generated by $\calA$-free sequences. In particular we are interested in establishing duality between such Young measures and $\calA$-quasiconvex functions, to be introduced shortly. 

\begin{defi}
We say that a family of probability measures $\mu = \{\mu_x\}_{x \in \Omega}$ is an $\calA$-$p$ Young measure  if $\mu$ is a Young measure generated by an $\calA$-free sequence of vector fields $V_j$ weakly convergent in $\LL^p$. If $\mu$ is a homogeneous Young measure, that is $\mu_x = \mu_{x_0}$ for almost every $x \in \Omega$ for some measure $\mu_{x_0}$ then we write $\mu \in \mathbb{H}^p_{\xi}$, where $\xi$ is the center of mass of our measure, i.e. $\xi = \langle \mu_{x_0}, \identity \rangle$.
\end{defi}

\begin{lemma}\label{lemmahpoclosed}
The set $\hpo$ is weak* closed in $E^*$.
\end{lemma}
\begin{proof}
This is part of the proof of Proposition 4.3 in \cite{FonsecaMuller99}. There it is stated that $\hpo$ is relatively closed in $\prob(\ccrd)$ with respect to the weak* topology on $E^*$. However it is easy to see that $\prob(\ccrd)$ itself is a weak* closed subset of $E^*$, thus proving our claim.
\end{proof}

\begin{defi}
For a measurable function $g \colon \ccrd \rightarrow \ccr$ satisfying $\left|g(v)\right| \leqslant C(1 + \left|v\right|^p)$ we define
$$ 
\qa g(v) := \inf \left\{ \inttorusn g(v + w(x))\dx \colon w \in C^{\infty}(\torusn) \cap \ker \calA, \, \inttorusn w \dx = 0 \right\} \ldotp
$$
\end{defi}

The above is intimately related to the notion of $\calA$-$p$ quasiconvexity. If $g$ is continuous then $C^{\infty}(\torusn) \cap \ker \calA$ in the above definition may be replaced by $\LL^p(\torusn) \cap \ker \calA$ without changing $\qa g$. When $g$ is such that $ g(v) \leqslant \inttorusn g(v + w(x))\dx$ for all $w \in \LL^p(\torusn) \cap \ker \calA$ we say that it is $\calA$-$p$ quasiconvex. This notion naturally corresponds to that of $\W1p$-quasiconvexity, as introduced in \cite{BallMurat84}. The fact that whether we take $C^{\infty}$ or $\LL^p$ does not make a difference under the growth assumptions on $g$, corresponds to the result by Ball and Murat (see \cite{BallMurat84}), which says that if $g$ is a continuous function satisfying an upper $p$-growth bound of the type $0 \leqslant g(\xi) \leqslant C(|\xi|^p + 1)$ then $g$ is $\W1p$-quasiconvex if and only if it is $\WW^{1, \infty}$-quasiconvex. The proof may easily be carried over to the general $\calA$-free framework using the decomposition results we gave earlier. What is perhaps more interesting is a negative result from the same paper that shows the importance of the $p$-growth bound. In Theorem 4.1 of \cite{BallMurat84} the authors exhibit an example of a function that is $\W1p$-quasiconvex if and only if the exponent $p$ is larger than the dimension of the space, thus showing that, in general, $\calA$-$p$ and $\calA$-$q$ quasiconvexity are different notions for $p \not= q$. This shows that caution must be exercised when dealing with such notions for potentially extended real-valued functions, which is what we aim to do in the later part of this paper. 

Before we do that, let us first use $\qa g$ to get a theoretical characterisation of $\calA$-$p$ Young measures, analogue of the one obtained by Kinderlehrer and Pedregal (see \cite{KinderlehrerPedregal94}) in the gradient case.

\begin{lemma}\label{lemmaqag}
For a continuous function $g \colon \ccrd \rightarrow \ccr$ satisfying $\left|g(v)\right| \leqslant C(1 + \left|v\right|^p)$ we have
$$ \qa (\qa g) = \qa g \ldotp$$
\end{lemma}

\begin{lemma}\label{lemmaPortmanteau}
If a sequence $\{ \nu_j \} \subset \prob(\ccrnn)$ converges to some $\nu \in \prob(\ccrnn)$ in the space $E^*$ then it also converges in the sense of weak convergence of probability measures. In particular, by Portmanteau's theorem, we have
$$ \int g d \nu_j \rightarrow \int g \dnux$$
for all bounded and continuous functions $g$, and
$$ \liminf_j \int g \mathop{d\nu_j} \geqslant \int g d\nu$$
for all lower semi-continuous functions $g$ bounded from below.
\end{lemma}
\begin{proof}
Immediately follows from bounded continuous functions being a subset of $E$.
\end{proof}

The two following results are essential for studying $\calA$-$p$ Young measures and are the last technical preliminaries we need. For proofs we refer to Theorem 4.1 in \cite{FonsecaMuller99}.
\begin{proposition}\label{propCharacterizationOfYM}
A probability measure $\mu \in \prob(\ccrNn)$ is a homogeneous $\calA$-free $\LL^p$ Young measure with mean $\xi_0$ if and only if $\mu$ satisfies $\intccrnn \xi \mathop{d\mu} = \xi_0$, $\int_{\ccrNn} \left|\xi\right|^p \mathop{d\mu}(\xi) < \infty$ and
$$ \int_{\ccrNn} g(\xi) \mathop{d\mu}(\xi) \geqslant \qa g(\xi_0) $$
for all $g \in E$.
\end{proposition}

A similar result holds for non-homogeneous Young measures. We state it in the following:
\begin{proposition}\label{propCharacterisationYMnonhom}
Fix $1 < p < \infty$ and let $\{ \nu_x\}_{x \in \Omega}$ be a weak* measurable family of probability measures on $\ccrd$. Then there exists a $p$-equiintegrable sequence $\{v_j\} \subset \LL^p(\Omega; \ccrd)$ generating the Young measure $\nu$ and satisfying $\calA v_j = 0$ if and only if the following conditions hold:

i) there exists $v \in \LL^p(\Omega; \ccrd) \cap \ker \calA$ such that
$$ v(x) = \langle \nu_x, \mathop{id} \rangle \text{ for a.e. } x \in \Omega;$$

ii) 
$$ \int_{\Omega} \intccrd |z|^p \dnux(z) \dx < \infty;$$

iii) for a.e. $x \in \Omega$ and all continuous functions $g$ satisfying $|g(v)| \leqslant C(1 + |v|^p)$ for some positive constant $C$ one has
$$ \langle \nu_x, g \rangle \geqslant \qa g( \langle \nu_x, \mathop{id} \rangle) \ldotp$$
\end{proposition}

\section{Closed $\calA$-$p$ quasiconvexity and lower semi-continuity}
Having at our disposal the relevant theory of $\calA$-free Young measures we are ready to state and prove our main result, which identifies the lower semi-continuous envelope of relaxation in a variational problem, where the lower semi-continuity is sequential, with respect to vector fields converging weakly in $\LL^p$ and strongly in $\Wm1p$ when the differential operator $\calA$ is applied. Such a result has already been proven in the special case of gradient (curl-free) vector fields in \cite{Kristensen15}, and we make extensive use of techniques present therein, adapting them as necessary to the case of a general operator $\calA$, similarly to \cite{FonsecaMuller99}. Here and in all that follows $p \in (1,\infty)$ is a fixed exponent. We begin with the necessary definitions.
 
\begin{defi}
We say that a function $f$ is closed $\calA$-$p$ quasiconvex if $f$ is lower semi-continuous and Jensen's inequality holds for $f$ and every homogeneous $\calA$-$p$ Young measure, i.e.
$$f(\xi) \leqslant \intccrd f(z) \mathop{d\nu(z)}$$
for every homogeneous $\calA$-$p$ Young measure $\nu$ with center of mass $\xi$. 
\end{defi}
\begin{defi}\label{defiQCenvelope}
For a measurable function $F$ we define its closed $\calA$-$p$ quasiconvex envelope by
$$ \overlineF(\xi) := \sup\{G(\xi) \colon G \leqslant F, \, G \text{ is closed } \calA -p \text{ quasiconvex} \} \ldotp$$
\end{defi}

\subsection{Sufficiency}
In all that follows we are interested in sequential lower semi-continuity of the functional
$$V \mapsto \II_F[V] := \intOmega F(V(x)) \dx,$$
with respect to vector fields $V_j \weakConv V$ weakly in $\LL^p$ and $\calA V_j \rightarrow \calA V$ strongly in $\Wm1p(\Omega)$. 
Our first result proves sufficiency of the closed $\calA$-$p$ quasiconvexity for the lower semi-continuity. 
\begin{theorem}\label{thmQCimpliesLSC}
Assume that $F$ is closed $\calA$-$p$ quasiconvex. Then the functional $\II_F$ is lower semi-continuous in the sense precised above. 
\end{theorem}
\begin{proof}
Fix an arbitrary sequence $\{V_j\} \subset \LL^p$ satisfying $ V_j \weakConv V$ in $\LL^p$ and $\calA V_j \rightarrow \calA V$ in $\Wm1p$. We wish to show that $\liminf_j \int_{\Omega} F(V_j) \dx \geqslant \intOmega F(V) \dx$. Without loss of generality we may assume that $\lim_j \int_{\Omega} F(V_j) \dx = \liminf_j \int_{\Omega} F(V_j) \dx$ and that $\{V_j\}$ generates some Young measure $\nu = \{ \nu_x \}_x$. If $\lim_j \int_{\Omega} F(V_j) \dx = \infty$ then there is nothing to prove, so assume $ 0 \leqslant \lim_j \int_{\Omega} F(V_j) \dx  < \infty$.
Since $F$ is a normal integrand we may use Theorem \ref{thmFToYM} to get
$$ \lim_j \int_{\Omega} F(V_j) \dx \geqslant \int_{\Omega} \int_{\mathbb{R}^{d}} F(\xi) \dnux(\xi) \dx,$$
hence it is enough to show
\begin{equation}\label{eqBeforeJensenForF} 
\int_{\mathbb{R}^{d}} F(\xi)  \dnux(\xi) \geqslant F(V(x))
\end{equation}
for almost all $x \in \Omega$. 
First observe that since $\{V_j\}$ is bounded in $\LL^p$ it is $\LL^1$-equiintegrable. Therefore, by Theorem \ref{thmFToYM}, we have
$$ V_j \weakConv \int_{\mathbb{R}^{d}} \xi \dnux(\xi) \quad \text{in } \LL^1\ldotp$$
By assumption we also have $V_j(x) \weakConv V(x)$ in $\LL^p$, so we must have, for almost every $x \in \Omega$ that $ V(x)  = \int_{\mathbb{R}^{d}} \xi \dnux(\xi) \ldotp$ Now the inequality \eqref{eqBeforeJensenForF} becomes
\begin{equation}\label{eqJensenForF}
\int_{\mathbb{R}^{d}} F(\xi)  \dnux(\xi) \dx \geqslant F( \int_{\mathbb{R}^{d}} \xi \dnux(\xi) ),
\end{equation}
and so is just Jensen's inequality for $F$ and $\nu_x$. Therefore it is enough to show that $\nu_x$ is an $\calA$-$p$ homogeneous Young measure for almost every $x \in \Omega$. 
First observe that the measure $\widetilde{\nu} := \{ \nu_x \ast \delta_{-V(x)} \}_x$ is generated by the sequence $V_j - V$, which is weakly convergent in $\LL^p$ and satisfies $\calA(V_j - V) \rightarrow 0$ in $\Wm1p$, so it is an $\calA$-$p$ Young measure. Now it is enough to apply Proposition \ref{propLocalisation} to deduce that for almost every $x$ the measure $\nu_x \ast \delta_{-V(x)}$ is an $\calA$-$p$ homogeneous Young measure, hence so is $\nu_x$ (as it is enough to add a constant to the sequence generating $\nu_x \ast \delta_{-V(x)}$ to generate $\nu_x$) and the proof is complete.
\end{proof}

\subsection{Quasiconvex envelope}
Before we move on to the (more difficult) results regarding necessity of closed $\calA$-$p$ quasiconvexity we need the following representation of a quasiconvex envelope of a function, similar to $\qa g$ introduced before, but without the upper growth bounds. An essential tool of this subsection is the classical measurable selection theorem due to Kuratowski and Ryll-Nardzewski (see \cite{KuratowskiRyllNardzewski65}), which follows. The standing assumption here and in all that follows is that $F \colon \ccrn \rightarrow [0, \infty]$ satisfies $F(\xi) \geqslant C|\xi|^p - C^{-1}$ for some constant $C > 0$ for all $\xi \in \ccrn$. In fact, since we only care about lower semi-continuity of the functional $\II_F$ we may without loss of generality assume that $F(\xi) \geqslant C|\xi|^p$, as adding the constant $C^{-1}$ to $F$ does not change the continuity properties of the functional.  

\begin{theorem}\label{thmKuratowski}
Let $X$ be a metric space and $Y$ be a separable and complete metric space. Fix a multi-valued function $G \colon X \rightarrow 2^Y$. If for any closed set $K \subset Y$ the set $\{ x \in X \colon F(x) \cap K \not= \emptyset \}$ is measurable then $G$ admits a measurable selector, i.e. there exists a measurable function $g \colon X \rightarrow Y$ such that for all $x \in X$ we have $g(x) \in G(x)$. 
\end{theorem}

Our goal is the following:
\begin{proposition}\label{lemmaQCenvelope}
The closed $\calA$-p quasiconvex envelope of a lower semi-continuous function $F$ satisfying the growth condition $F(\xi) \geqslant c|\xi|^p$ is given by
$$ \overlineF (\xi) = \inf_{\nu \in \hpo} \langle F(\cdot + \xi), \nu \rangle = \inf_{\nu \in \mathbb{H}^p_{\xi}} \langle F, \nu \rangle \ldotp
$$
Moreover, the function $\overline{F}$ is indeed closed $\calA$-$p$ quasiconvex. 
\end{proposition}
\begin{proof}
Denote 
$$ R(\xi) := \inf_{\nu \in \hpo} \langle F(\cdot + \xi), \nu \rangle \ldotp$$
Clearly for any $\nu \in \hpo$ and $\xi \in \ccrNn$ we have $ \overlineF(\xi) \leqslant \langle F(\cdot + \xi), \nu \rangle,$ therefore taking the infimum over $\nu \in \hpo$ yields
$$ \overlineF(\xi) \leqslant R(\xi),$$
hence showing that $R$ is closed $\calA-p$ quasiconvex will give the opposite inequality and end the proof. 

To show that $R$ is lower semi-continuous fix $\xi_0 \in \ccrNn$ and a sequence $\xi_j \rightarrow \xi_0$ and an $\varepsilon > 0$. We will show that 
$$ \varepsilon + \liminf R(\xi_j) \geqslant R(\xi_0) \ldotp$$
Without loss of generality assume that $ \lim R(\xi_j) = \liminf R(\xi_j) < \infty,$
and let $M$ be such that $R(\xi_j) + \varepsilon \leqslant M$ for all $j$. By definition of $R$ for each $\xi_j$ there exists $\nu_j \in \hpo$ with
$$ M \geqslant R(\xi_j) + \varepsilon \geqslant \langle F(\cdot + \xi_j), \nu_j \rangle \ldotp$$
Our growth assumption on $F$ and boundedness of $\left|\xi_j\right|$ (as a convergent sequence) then imply
$$ M \geqslant \intccrnn c \left| \xi + \xi_j \right|^p \mathop{d \nu_j} \geqslant C \left( \intccrnn \left|\xi\right|^p \mathop{d\nu_j} -1 \right),$$
which yields $ \sup_j \intccrnn \left|\xi\right|^p \mathop{d\nu_j} < \infty \ldotp$
We see that the family $\{ \nu_j \}$ is bounded in $E^*$, therefore we may extract a weakly*-convergent subsequence from it - without loss of generality assume that the whole sequence converges, i.e. $\nu_j \weakstarconv \nu_0$. By Lemma \ref{lemmahpoclosed} we have $\nu_0 \in \hpo$. Moreover
$ \translation_{\xi_j} \ast \nu_j \weaksconv \translation_{\xi_0} \ast \nu_0 \ldotp$
Since $F$ is lower semi-continuous and bounded from below we have
$$ \varepsilon + \liminf R(\xi_j) \geqslant \liminf \langle F, \translation_{\xi_j} \ast \nu_j \rangle \geqslant \langle F, \translation_{\xi_0} \ast \nu_0 \rangle = $$
$$ = \intccrnn F(\cdot + \xi_0) \mathop{d\nu_0} \geqslant R(\xi_0),$$
where the last inequality comes from the definition of $R$ and the fact that $\nu \in \hpo$. Since $\varepsilon > 0$ was arbitrary we conclude that $R$ is in fact lower semi-continuous. 

It now remains to show that $R$ satisfies the Jensen's inequality with respect to $\calA$-free measures. To that end fix $\xi_0 \in \ccrnn$ and $\nu \in \mathbb{H}^p_{\xi_0}$. We wish to show that
$R(\xi_0) \leqslant \intccrnn R \dnu \ldotp$
Observe that we may assume without loss of generality that $\intccrnn R \dnu < \infty$, as the case where this integral is infinite is trivial. Let us fix an $\varepsilon > 0$ and observe that, by definition of $R$, for all $\xi \in \ccrnn$ there exists $\nuxi \in \hpo$ satisfying
$$ \langle F(\cdot + \xi), \nuxi \rangle \leqslant \varepsilon + R(\xi),$$
so that
$$\intccrnn \left( \intccrnn F(\cdot + \xi) \mathop{d\nuxi} \right) \dnu(\xi) \leqslant \varepsilon + \intccrnn R \dnu \ldotp$$
Now - if we manage to show that $\nuxi$ may be chosen in such a way that $\xi \mapsto \nuxi$ is weak* measurable and that the measure $\mu$ defined by duality as
\begin{equation}\label{eqMuDefDuality}
\langle g, \mu \rangle := \intccrnn \left( \intccrnn g(\cdot + \xi) \mathop{d\nuxi} \right) \mathop{d\nu}(\xi)
\end{equation}
is an $\calA-p$ homogeneous Young measure with mean $\xi_0$ then the claim will follow, as by definition $\langle F, \mu \rangle \geqslant R(\xi_0)$. 

For the measurable selection part we define a multifunction $\calF$ given by
$$ \calF(\xi) := \left\{ \mu \in \hpo \colon \intccrnn F(\cdot + \xi) \mathop{d\mu} \leqslant \varepsilon + R(\xi) \right\} \ldotp$$
For the measurable selection result we intend to use (see Theorem \ref{thmKuratowski}) we need $\calF$ to take values in $2^Y$ for some complete metric space $Y$. For that we define, for a given $M > 0$,
$$ \Omega_M := \{ \xi \in \ccrnn \colon \left|\xi\right| < M, R(\xi) \leqslant M \} \ldotp $$
Observe that since we assumed $R$ to be integrable with respect to $\nu$, we have $\xi \in \bigcup_{M = 1}^{\infty} \Omega_M$ for $\nu$-a.e. $\xi \in \ccrnn$. Let us fix $M \in \naturals$. Then, for any $\xi \in \Omega_M$ and any $\mu \in \calF(\xi)$, we have
$$ \int F(z + \xi) \mathop{d\mu}(z) \leqslant \varepsilon + R(\xi) \leqslant 2 \varepsilon + R(\xi) \leqslant M+2 \varepsilon \ldotp$$
The factor $2$ in front of $\varepsilon$ is not important here, we only put it there to allow for some room in the later part of the argument. Due to the growth assumption on $F$ there holds
$$ \int F(z + \xi) \mathop{d\mu}(z) \geqslant C \int \left|z + \xi\right|^p \mathop{d\mu}(z) \geqslant C \int \left|z\right|^p \mathop{d\mu}(z) - C^{-1} \left|\xi\right|^p \ldotp$$
Finally $  \int \left|z\right|^p \mathop{d\mu}(z) \leqslant C_M,$ holds for all $\mu \in \calF(\xi)$, with the constant $C_M$ depending only on $M$ (and $\varepsilon$). Therefore we may consider our operator $\calF$ as a map $\Omega_M \rightarrow 2^{Y_M}$, where
$$ Y_M := \left\{ \mu \in \hpo \colon \int \left|z\right|^p \mathop{d\mu} \leqslant C_M \right\} \ldotp$$
The set $Y_M$ may be equipped with the weak* topology inherited from $E^*$. Since we put a uniform bound on the $p$-th moments (so also on the norm in $E^*$), this topology is metrisable in a complete and separable manner. To prove that first recall that due to Lemma \ref{lemmahpoclosed} $\hpo$ is weak* closed in $E^*$. Since $| \cdot |^p \in E$ we know that the map $\mu \mapsto \intccrd |z|^p \mathop{d\mu}$ is weak* continuous, thus $Y_M$ is weak* closed and bounded. The Banach-Alaoglu Theorem (see for example Theorem 3.16 in \cite{Brezis10}) then implies that $Y_M$ is weak* compact. Since $E$ is clearly separable we deduce that the weak* topology on $Y_M$ is metrisable (see Theorem 3.28 in \cite{Brezis10}). Finally, compact metric spaces are complete and separable, thus proving our claim. 
\begin{lemma}\label{lemmaMeasurability1}
For any $\xi \in \Omega_M$ the set $\calF(\xi)$ is non-empty and closed.
\end{lemma}
\begin{proof}
The fact that $\calF(\xi) \not= \emptyset$ comes straight from the definition of $R$. To show that it is closed it is enough to show that it is sequentially closed. Let us then fix a sequence $\{\mu_j\} \subset \calF(\xi)$ and assume that it converges weak* in $E^*$ to some $\mu \in Y_M$. Since the function $F$ is lower semi-continuous and bounded from below we get by Lemma \ref{lemmaPortmanteau} that
$$ R(\xi) + \varepsilon \geqslant \liminf \int F(\cdot + \xi) \mathop{d\mu_j} \geqslant  \int F(\cdot + \xi) \mathop{d\mu},$$
so $\mu \in \calF(\xi)$, which ends the proof.
\end{proof}

\begin{lemma}\label{lemmaMeasurability2}
For any non-empty closed set $O \subset Y_M$ the set $\{ \xi \in \Omega_M \colon \calF(\xi) \cap O \not= \emptyset \}$ is (Lebesgue) measurable. 
\end{lemma}
\begin{proof}
First we write
$$\{ \xi \in \Omega_M \colon \calF(\xi) \cap O \not= \emptyset \} = \bigcap_{k=1}^{\infty} \left\{ \xi \in \Omega_M \colon \inf_{\mu \in O} \int F(\cdot + \xi) \mathop{d\mu} \leqslant R(\xi) + \varepsilon(1 + 2^{-k}) \right\}\ldotp$$
Hence it is enough to show that the sets 
$$\left \{ \xi \in \Omega_M \colon \inf_{\mu \in O} \int F(\cdot + \xi) \mathop{d\mu} \leqslant R(\xi) + \varepsilon(1 + 2^{-k}) \right\}$$ 
are all measurable. Define
$$ U(\xi) := \inf_{\mu \in O} \int F(\cdot + \xi) \mathop{d\mu}\ldotp $$
We claim that $U$ is lower semi-continuous. Let $\xi_j \rightarrow \xi$. We need to show that $ \liminf_j U(\xi_j) \geqslant U(\xi) \ldotp$
Without loss of generality $\lim_j U(\xi_j) = \liminf_j U(\xi_j) < \infty$. By definition of $U$ for each $k$ there exists a measure $\mu_j \in O$ with
$$ \int F(\cdot + \xi_j) \mathop{d \mu_j} \leqslant U(\xi_j) + 1/k \ldotp$$
Therefore
$$ \lim_j \int F(\cdot + \xi_j) \mathop{d \mu_j} = \lim_j U(\xi_j) \ldotp$$
Since the set $O$ is a closed subset of a compact space $Y_M$ we may extract an $E^*$ weak* convergent subsequence from $\mu_j$. Without loss of generality assume that the entire sequence $\mu_j$ converges weak* to some $\mu \in O$. This, combined with $\xi_j \rightarrow \xi$, implies that we have
$$ \delta_{\xi_j} \ast \mu_j \weakstarconv \delta_{\xi} \ast \mu$$
in the sense of probability measures. Therefore, since $F$ is lower semi-continuous, Portmanteau's theorem yields
$$ \liminf_j \int F(\cdot + \xi_j) \mathop{d \mu_j}  = \liminf_j \int F d \delta_{\xi_j} \ast \mu_j  \geqslant  \int F d \delta_{\xi} \ast \mu =  \int F(\cdot + \xi) \mathop{d \mu} \geqslant U(\xi),$$
which shows that $U$ is indeed lower semi-continuous. Since
$$ \left\{ \xi \in \Omega_M \colon \inf_{\mu \in O} \int F(\cdot + \xi) \mathop{d\mu} \leqslant R(\xi) + \varepsilon(1 + 2^{-k}) \right\} = \{ \xi \in \Omega_M \colon U(\xi) \leqslant R(\xi) + \varepsilon(1 + 2^{-k}) \}$$
and both $U$ and $R$ are lower semi-continuous (hence measurable) the set in question is measurable as well, which ends the proof.
\end{proof}

Now, thanks to Lemmas \ref{lemmaMeasurability1} and \ref{lemmaMeasurability2} we may use Theorem \ref{thmKuratowski} to deduce the existence of a measurable map $\nu^M \colon \Omega_M \rightarrow \hpo$ such that for any $\xi \in \Omega_M$ the measure $\nu^M_{\xi}$ satisfies
$$ \intccrnn F(\cdot + \xi) d\nu^M_{\xi} \leqslant \varepsilon + R(\xi) \ldotp$$
Finally let us define the map $\widetilde{\nu} \colon \ccrnn \rightarrow \hpo$ by
\begin{equation*} 
\widetilde{\nu}_\xi := \begin{cases}
\nu^M_{\xi} \text{ for } \xi \in \Omega_M \setminus \Omega_{M-1} \\
\widetilde{\mu} \text{ for } \xi \not\in \bigcup_{M=1}^{\infty} \Omega_M,
\end{cases}
\end{equation*}
where $\widetilde{\mu}$ is some arbitrary element of the (non-empty) set $\hpo$. Observe that the choice of $\widetilde{\mu}$ does not matter, as we have already observed that the set $\ccrnn \setminus \bigcup_{M=1}^{\infty} \Omega_M$ is of $\nu$ measure $0$. Clearly the map $\widetilde{\nu}$ is weak* measurable, so we may define $\mu \in (C_0(\ccrnn))^*$ as in \eqref{eqMuDefDuality}. It only remains to show that $\mu \in \hpo$. 

Positivity of $\mu$ results immediately from positivity of all $\nu_{\xi}$ and $\nu$. In the same way we show that $\mu$ is a probability measure, as
$$\langle 1, \mu \rangle =  \intccrnn \left( \intccrnn 1 \mathop{d\nuxi} \right) \dnu(\xi) =   \intccrnn 1 d\nu(\xi) = 1,$$
since all measures considered are probability measures. To prove that $\mu$ has a finite $p$-th moment we write
$$\langle \left|\cdot\right|^p, \mu \rangle = \intccrnn \left( \intccrnn \left|\cdot + \xi\right|^p \mathop{d\nuxi} \right) \dnu(\xi) \ldotp $$
Using the growth assumption on $F$ we get
$$ \intccrnn \left|\cdot + \xi\right|^p \mathop{d\nuxi} \leqslant C  \intccrnn F(\cdot + \xi) \mathop{d\nuxi} \leqslant C (R(\xi) + \varepsilon),$$
where the last inequality is satisfied for $\nu$-a.e. $\xi$. Integrating with respect to $\nu$ gives
$$\langle \left|\cdot\right|^p, \mu \rangle \leqslant C \left(\varepsilon + \intccrnn R(\xi) \dnu(\xi)\right) < \infty,$$
since, by assumption, $R$ is integrable with respect to $\nu$. Lastly, it remains to show that $\mu$ satisfies the inequality in \ref{propCharacterizationOfYM}. Fix any admissible test function $g \in E$. We have
\begin{eqnarray*}
\langle \mu, g \rangle &=&  \intccrnn \left( \intccrnn g(\cdot + \xi) \mathop{d\nuxi} \right) \dnu(\xi) \\
&\geqslant&  \intccrnn \qa g(\xi) \dnu(\xi) \geqslant \qa (\qa g) (\xi_0) = \qa g(\xi_0),
\end{eqnarray*} 
where the first inequality comes from the fact that all $\nuxi$'s are Young measures with mean $0$, the second one from the respective property of $\nu$, and the last equality from Lemma \ref{lemmaqag}. This shows that we indeed have $\mu \in \mathbb{H}^p_{\xi_0}$ and ends the proof, as discussed in \eqref{eqMuDefDuality}.
\end{proof}

\subsection{Necessity}
We are now ready to state and prove the main result of the paper:
\begin{theorem}\label{thmContinuousImpliesRelaxation}
If $F \colon \ccrd \to (-\infty, \infty]$ is a continuous integrand satisfying $F(\xi) \geqslant C|\xi|^p - C^{-1}$ for some $C > 0$ then the lower semi-continuous envelope of the functional $\II_F$ is given by
$$\overline{\II}_F [V] := \inf_{V_j} \left\{ \liminf_j \II_F [V_j] \right\}= \int_{\Omega} \overline{F}(V(x)) \dx,$$
where the infimum is taken over all admissible test sequences, i.e. satisfying $V_j \weakConv V$ weakly in $L^p$ and $\calA V_j \rightarrow \calA V$ strongly in $\Wm1p$. As before, $\overline{F}$ denotes the closed $\calA$-$p$ quasiconvex envelope of $F$. 
\end{theorem}

\begin{proof}
Theorem \ref{thmQCimpliesLSC} guarantees that $\overline{\II}_F [V] \leqslant \int_{\Omega} \overline{F}(V(x)) \dx$, thus we only need to prove the opposite inequality. If $F$ is identically equal $+ \infty$ then there is nothing to show, so we may restrict to proper integrands.
As before we may assume $F(\xi) \geqslant C|\xi|^p$. Fix any $V \in \LL^p$. Without loss of generality we may assume  $\int_{\Omega} \overline{F}(V(x)) \dx < \infty$, as otherwise there is nothing to prove. Fix an $\varepsilon > 0$ and observe that clearly we must have $ \overline{F}(V(x)) < \infty \text{ a.e. in } \Omega \ldotp$
Therefore, using Proposition \ref{lemmaQCenvelope}, we may find a family of homogeneous $\calA$-$p$ Young measures $\{\nu_x\}_{x \in \Omega}$ with mean $0$ and such that, for almost every $x \in \Omega$, we have
\begin{equation}\label{eqNuxFbar} 
\overline{F}(V(x)) + \varepsilon \geqslant \intccrd F(\cdot + V(x)) d\nu_x \ldotp
\end{equation}
Using exactly the same argument as in the proof of Proposition \ref{lemmaQCenvelope} we may ensure weak* measurability of $x \rightarrow \nu_x$. We intend to show that $\nu$ is a suitable Young measure using Proposition \ref{propCharacterisationYMnonhom}. The first point therein is clearly satisfied, as all our measures are of mean $0$. The second one may be checked in the same way as in the already mentioned proof of Proposition \ref{lemmaQCenvelope}, using the growth assumption on $F$. Finally, the third point results immediately from the fact that all $\nu_x$'s are, by definition, elements of $\hpo$, so we may use Proposition \ref{propCharacterizationOfYM}. This shows that $\nu$ is indeed generated by some $p$-equiintegrable family $\{W_j\} \subset \LL^p(\Omega; \ccrd) \cap \ker \calA$ with $W_j \weakConv 0$ in $\LL^p$. For a given $M \in \naturals$ consider $F^M(z) := \min(F(z), M(|z|^p + 1))$.
Clearly, for each $M$, the function $F^M$ is continuous and the family $\{F^M(V + W_j)\}_j$ is $p$-equiintegrable, due to the same property of $\{V + W_j\}$. Theorem \ref{thmFToYM} then yields
$$ \int_{\Omega} F^M(V + W_j) \dx \rightarrow \int_{\Omega} \left( \intccrd F^M(V(x) + \cdot) \dnux \right) \dx \ldotp$$
On the other hand, since $F^M \leqslant F$ and $\nu_x$ are non-negative and satisfy \eqref{eqNuxFbar}, we have
\begin{eqnarray*}
\int_{\Omega} \left( \intccrd F^M(V(x) + \cdot) \dnux \right) \dx &\leqslant& \int_{\Omega} \left( \intccrd F(V(x) + \cdot) \dnux \right) \dx\\ 
&\leqslant& \int_{\Omega} \overline{F}(V(x)) \dx + \varepsilon \ldotp
\end{eqnarray*}
From this we deduce, through a diagonal extraction, that there exists a sequence $j(M) \in \naturals$ with $\lim_{M \rightarrow \infty} j(M) = \infty$ such that for all $M$ one has
\begin{equation}\label{eqEstimateForFM} 
\int_{\Omega} F^M(V + W_{j(M)}) \dx \leqslant \int_{\Omega} \overline{F}(V(x)) \dx + 2 \varepsilon \ldotp
\end{equation}
Define the set
$$ \gm := \left\{ x \in \Omega \colon F(V(x) + W_j(x)) \leqslant M(|V(x) + W_j(x)|^p + 1) \right\},$$
and fix some $\xi_0 \in \ccrd$ for which $F(\xi_0) < \infty$, which exists, as $F$ is proper. Next define a vector field $\wntilde$ in such a way that
\begin{equation}\label{eqDefWnTilde} 
V(x) + \wntilde(x) = (V(x) + W_{j(M)}(x)) \indyk_{\gm} + \xi_0 \indyk_{\gm^c} \ldotp
\end{equation}
We claim that $\{V + \wntilde\}_M$ is an admissible vector field in the $\overline{\II}[V]$ problem.
For that it is enough to show that $\| V + \wntilde - (V + W_{j(M)}) \|_{\LL^p(\Omega)} \rightarrow 0$. 
By definition we have 
$$\| V + \wntilde - (V + W_{j(M)}) \|_{\LL^p(\Omega)} = \| V + \wntilde - (V + W_{j(M)}) \|_{\LL^p(\gm^c)} \leqslant $$
$$ \leqslant \|\xi_0\|_{\LL^p(\gm^c)} + M^{-1} \left(\intOmega F^M(V + W_{j(M)}) \dx \right)^{1/p},$$
where the last inequality comes from the definition of the set $\gm^c$ (and extending the integral to all of $\Omega$). Now, the last term here is bounded by $M^{-1} \left( \int_{\Omega} \overline{F}(V(x)) \dx + 2 \varepsilon \right)^{1/p}$ due to \eqref{eqEstimateForFM}, thus showing the desired convergence to $0$ in $\LL^p$, as $\|\xi_0\|_{\LL^p(\gm^c)} \to 0$ results simply from the fact that clearly the Lebesgue measure of $\gm^c$ tends to $0$.
This implies in particular that $\calA \wntilde \rightarrow 0$ in $\Wm1p$ and $\wntilde \weakConv 0$ in $\LL^p$. Therefore if we define
$$ V_M(x) := V(x) + \wntilde(x)$$
we see that
\begin{equation*}
\begin{cases}
V_M \weakConv V \quad \text{in } \LL^p,\\
\calA V_M \rightarrow \calA V \quad \text{in } \Wm1p \ldotp
\end{cases}
\end{equation*}
Thus
\begin{eqnarray*}
\overline{\II}[V] &\leqslant& \liminf_{M \rightarrow \infty} \int_{\Omega} F(V + \wntilde) \dx\\ 
&=& \liminf_{M \rightarrow \infty} \int_{\gm} F^M(V + W_{j(M)}) \dx + \int_{\gm^c} F(\xi_0) \dx\\
&\leqslant& \liminf_{M \rightarrow \infty} \int_{\Omega} \overline{F}(V(x)) \dx + 2 \varepsilon = \int_{\Omega} \overline{F}(V(x)) \dx + 2 \varepsilon,
\end{eqnarray*} 
where the last inequality results from \eqref{eqEstimateForFM} and the measure of $\gm^c$ tending to $0$. Since $\varepsilon > 0$ was arbitrary the proof is complete.
\end{proof}

\begin{remark}
Observe that the above result is stronger than just necessity of quasiconvexity for lower semi-continuity. The downside is the continuity requirement for the integrand. However, it seems that this assumption cannot be easily removed if one hopes for a full relaxation result in the spirit of the one above. This has been discussed in \cite{BraidesFonsecaLeoni00} in an example given in Remark 1.2. There the authors exhibit an example of a constant rank operator $\calA$ operating on vector fields $v \colon \ccr \to \ccr^2$ with $\calA v = 0$ if and only if $v_2' = 0$ and a discontinuous function $f$ defined as
\begin{equation*}
f(v) := 
\begin{cases}
(v_1 - 1)^2 + v_2^2, \quad \text{if } v_2 \in \mathbb{Q},\\
(v_1 + 1)^2 + v_2^2, \quad \text{if } v_2 \in \ccr \setminus \mathbb{Q} \ldotp
\end{cases}
\end{equation*}
Observe that $f$ satisfies quadratic growth bounds both from above and from below. Nevertheless, it turns out that if $\calF(v; (a,b))$ denotes the sequential lower semi-continuous envelope of $v \mapsto \int_a^b f(v) \dx$ with respect to $v_j \weakConv v$ in $\LL^p$ and $\calA v_j \to \calA v$ in $\Wm1p$ then for any interval $(a,b) \subset (0,1)$ one has
$$ \calF(v; (a,b)) = \min \left( \int_a^b (v_1 - 1)^2 + v_2^2 \dx, \int_a^b (v_1 + 1)^2 + v_2^2 \dx \right) \ldotp$$
Thus, $\calF(v; \cdot)$ is not the trace of a Radon measure, hence there cannot exist an integrand $F$ such that $\calF(v; (a,b)) = \int_a^b F(v) \dx$ and so in general one cannot hope for a relaxation result of the type above that allows for discontinuous functions.
\end{remark}

This being said, it is still possible to obtain some results for less regular integrands, and this is what we will do in the last part of the paper. We show that under additional conditions on the characteristic cone of the operator $\calA$ the continuity of the integrand may be deduced from lower semi-continuity of the functional, and thus need not be assumed, hence leading to the equivalence of sequential weak lower semi-continuity of the functional and closed $\calA$-$p$ quasiconvexity of the integrand. This will be the content of our final result.

\begin{defi}
We define the characteristic cone of $\calA$ to be the set 
$$ \Lambda := \bigcup_{w \in S^{N-1}} \ker \bbA (w) \ldotp$$
\end{defi}

\begin{lemma}\label{lemmaSepConvex}
Suppose that the integrand $F$ is real-valued and such that the functional $V \mapsto \II_F[V]$ is $\calA$-$\infty$ sequentially weakly* lower semi-continuous, i.e. for every sequence $V_j$ with $V_j \weakStarConv V$ in $\LL^{\infty}(\Omega; \ccrd)$ and $\calA(V_j - V) = 0$ for all $j$ one has
$$\int_{\Omega} F(V(x)) \dx \leqslant \liminf_{j \rightarrow \infty} \int_{\Omega} F(V_j(x)) \dx \ldotp$$
Then $F$ is (separately) convex along any direction given by a vector in $\Lambda$.
\end{lemma}
Similar results have been given in the literature, for example in Section 6 of \cite{Tartar79}, but for the sake of completeness we present a proof for this particular case below.
\begin{proof}
Fix $\theta \in (0,1)$ and $y,z \in \ccrd$ such that $y - z \in \ker \bbA (w) \subset \Lambda$ with $w \in S^{N-1}$. We need to show that $ F(\theta y + (1 - \theta) z) \leqslant \theta F(y) + (1-\theta)F(z) \ldotp$
To this end let $Q_w \subset \ccrN$ be a rotated unit cube with two neighbouring (connected by an edge) vertices $0$ and $w$. Such a cube is not unique, but that is not important, simply pick an arbitrary one. Define a function $u$ by
\begin{equation*}
u(x) := \begin{cases} (1-\theta)(y-z) \colon \langle x, w \rangle \in [0, \theta), \\
\theta(z-y) \colon \langle x, w \rangle \in [\theta, 1]
\end{cases}
\end{equation*}
for $x \in Q_w$. Thus $\left| \{ x \colon u(x) = (1-\theta)(y-z) \} \right| = \theta$ and $\left|  \{ x \colon u(x) = \theta(z-y) \} \right| = 1 - \theta$. Extend $u$ to $\ccrN$ by $Q_w$-periodicity. Since $y-z \in \ker \bbA(w)$ it is easy to see that $\calA u = 0$.
Finally let $u_j(x) := u(nx)$ for $x \in \Omega$. By Lemma \ref{lemmaOscillationConv} we have
$$ u_j \weakStarConv \int_{Q_w} u(y) \dy = 0,$$
where the convergence is weak* in $\LL^{\infty}(\Omega)$. Clearly we also have $\calA u_j = 0$ for all $n$. Thus we may use the lower semi-continuity assumption on our functional with $V(x) := \theta y + (1-\theta)z$ and $V_j(x) := V + u_j$. This yields
\begin{eqnarray*}
\left| \Omega \right| F(\theta y + (1-\theta)z ) &=& \int_{\Omega} F(V(x)) \dx \leqslant \liminf_j \int_{\Omega} F(V(x) + u_j(x)) \dx \\
&=&  \left| \Omega \right| \left( \theta F(y) + (1-\theta) F(z) \right),
\end{eqnarray*}
thus ending the proof.
\end{proof}

\begin{corollary}\label{corContinuityFromLsc}
Suppose that the characteristic cone of $\calA$ spans the entire space, i.e. $\span2 \Lambda = \ccrn$ and that $F$ is real-valued and such that $\II_F$ is lower semi-continuous in the sense of the previous Lemma. Then $F$ is continuous.
\end{corollary}
\begin{proof}
Using the previous Lemma and the assumption $\span2 \Lambda = \ccrn$ one may show that $F$ is locally Lipschitz in the exact same manner as for rank-one convex functions. We refer the reader to  \cite{BallKircheimKristensen00} for details.

Let us note that the assumption $\span2 \Lambda = \ccrn$ is essential here. When it fails, the $\calA$-quasiconvexity does not improve regularity along directions that are not in $\span2 \Lambda$, in fact there are examples for loss of regularity, when taking the $\calA$-quasiconvex envelope of a smooth function yields a discontinuous one - see Remark 3.5 in \cite{FonsecaMuller99}.
\end{proof}

This leads to the final result of the paper:
\begin{theorem}\label{thmLSCimpliesQC}
Suppose that $F$ is real valued, satisfies the growth condition $F(\xi) \geqslant C |\xi|^p$ and that $\span2 \Lambda = \ccrn$. Then the functional $\II_F$ is sequentially lower semi-continuous in the usual sense if and only if $F$ is closed $\calA$-$p$ quasiconvex. 
\end{theorem}
\begin{proof}
We already know that closed $\calA$-$p$ quasiconvexity implies lower semi-continuity of the functional, thanks to Theorem \ref{thmQCimpliesLSC}. The other implication is a simple consequence of Corollary \ref{corContinuityFromLsc} and Theorem \ref{thmContinuousImpliesRelaxation}. If we assume lower semi-continuity of the functional then this implies that the relaxation introduced in Theorem \ref{thmContinuousImpliesRelaxation} is equal to the functional itself. On the other hand, continuity of the integrand implies that the relaxation is given by integration of the quasiconvex envelope. From these two facts we infer that $F$ must be equal to its quasiconvex envelope, thus ending the proof.
\end{proof}

\section{Appendix}
We have mentioned in the introduction that the regularisation results for sequences generating $\calA$-$p$ Young measures rely heavily on an analogue of Helmholtz decomposition for the operator $\calA$. Specifically one wants to obtain a projection-like operator onto the kernel of $\calA$, show that it is a Fourier multiplier, and then do the same for a generalised inverse of the projection. This is done using the following (see \cite{SteinWeiss71}) result:

\begin{proposition}\label{propFourierMultipliers}
If $\Theta \colon \ccrN \setminus \{0\} \rightarrow \ccr$ is homogeneous of degree $0$ and if it is smooth on $S^{N-1}$ then the operator $T_{\Theta} \colon \LL^p(\torusN) \rightarrow \LL^p(\torusN)$ defined by
$$T_{\Theta}f(x) :=  \sum_{\lambda \in \Delta \setminus \{0\} } \Theta(\lambda) \widehat{f}(\lambda) e^{2\pi i x\cdot \lambda} \quad \text{for } f \in \LL^p(\torusN), \, f = \sum_{\lambda \in \Delta} \widehat{f}(\lambda)e^{2\pi i x\cdot \lambda} $$
is a Fourier multiplier operator for any $1 < p < \infty$. 
\end{proposition}
Here $\Delta := \mathbb{Z}^N \subset \ccrN$, and for $\lambda \in \Delta$ we denote by $\widehat{f}(\lambda)$ the corresponding Fourier coefficient of the function $f$. 
The projection and its generalised inverse are denoted $\bbP(w)$ and $\bbQ(w)$ respectively and, for a given $w \in \ccrN$, are defined as follows. The projection $\bbP(w)$ is simply the orthogonal projection of $\ccrn$ onto $\ker( \bbA(w))$, whilst $\bbQ(w)$ is required to satisfy $ \bbQ(w) \equiv 0 \text{ on } \im(\bbA(w))^{\perp}$, and for $\bbA(w)v \in \im(\bbA(w))$ with $v \in \ccrd$
\begin{equation}\label{eqQAP}
\bbQ(w)(\bbA(w)v) = v - \bbP(w)v \ldotp
\end{equation}
That is, $\bbQ(w)$ is the Moore-Penrose generalised inverse of $\bbP(w)$. 

The difficult part here is showing smoothness of the maps $w \mapsto \bbP(w)$ and $w \mapsto \bbQ(w)$ in order to be able to use the previous Proposition. This has been argued previously by means of the Cauchy Representation Formula and the usual reference given is \cite{Kato13}. However we were unable to find a full, detailed proof of this result, thus below we offer a more elementary one, relying on early results on the Moore-Penrose generalised inverse.

The main source for this part is \cite{Evard90} with the exception of the very first lemma we give, which follows another paper by the same author, see Lemma 5.4 \cite{Evard85}.

\begin{lemma}\label{lemmaLinearIndependence}
Let $D \subset \ccrN$ be open with $w_0 \in D$. Let $x_1, \ldots x_k$ be continuous functions defined on $D$ and valued in $\ccrm$. Assume that the vectors $x_1(w_0), \ldots, x_k(w_0)$ are linearly independent. Then there exists some open neighbourhood $D_0$ of $w_0$ such that $x_1(w), \ldots, x_k(w)$ are linearly independent for all $w \in D_0$. 
\end{lemma}

The following two results correspond, in that order, to Proposition 3.1 and Corollary 3.4 in \cite{Evard90}.

\begin{lemma}\label{lemmaGramSchmidt}
Let $D \subset \ccrN$ be open. Suppose that $a_1, \ldots, a_k \in C^{\infty}(D; \ccrm)$ are such that $a_1(w), \ldots, a_k(w)$ are linearly independent for each $w \in D$. Then there exists a unique family of vector functions $u_1, \ldots u_k \in C^{\infty}(D; \ccrm)$ such that for each $w \in D$ the family $u_1(w), \ldots, u_k(w)$ is orthonormal, and for each $j \in \{1, \ldots, k\}$ one has
$$ \span2 \{a_1(w), \ldots, a_j(w)\} =  \span2 \{u_1(w), \ldots, u_j(w)\} \ldotp$$
\end{lemma}

\begin{lemma}\label{lemmaLocalRankDecomposition}
Let $D \subset \ccrN$ be open and let $r \in \naturals$. Suppose that $A$ is a matrix-valued function of class $C^{\infty}$ and constant rank $r$, i.e. $A \in C^{\infty}(D; \ccrmn_r)$. Fix any $w_0 \in D$. Then there exist an open neighbourhood $D_0 \subset D$ of $w_0$ and functions $U \in  C^{\infty}(D_0; \ccrmr_r)$ and $B \in  C^{\infty}(D_0; \ccrrn_r)$  such that $U^T (w) U(w) \equiv \identity_r$ and
$$A(w) = U(w)B(w) \quad \text{ for all } w \in D_0 \ldotp$$
\end{lemma}

This full-rank decomposition in Lemma \ref{lemmaLocalRankDecomposition} may be used to determine the Moore-Penrose generalised inverse of a matrix, as in Theorem 5, Chapter 1 of \cite{BenIsrael03} (originally due to MacDuffee). 

\begin{theorem}\label{thmInverseRepresentation}
If $A \in \ccrmn_r$ with $r>0$ has a full-rank factorisation $A = UB$ satisfying $U \in \ccrmr_r$, $B \in \ccrrn_r$ and $U^T U = \identity_r$ then
$$ \amp = B^T(BB^T)^{-1}U^T \ldotp$$
\end{theorem}

Combining all these we may now easily prove the following:
\begin{theorem}\label{thmSmoothInverse}
Let $D \subset \ccrN$ be open. Assume that $A \colon D \rightarrow \ccrmn_r$ is a constant-rank matrix valued function of class $C^{\infty}$. Then the function $\amp \colon D \rightarrow \ccrn \times m_r$, given by taking the Moore-Penrose generalised inverse of $A(w)$ at each point $w \in D$, is of class $C^{\infty}$ as well. 
\end{theorem}
\begin{proof}
Clearly it is enough to show this result locally. First of all, if the rank $r = 0$ then $A \equiv 0$, in which case the result is trivial. Assuming $r > 0$ we fix an arbitrary $w_0 \in D$. Then Lemma \ref{lemmaLocalRankDecomposition} yields a local full-rank decomposition
$$ A(w) = U(w)B(w),$$
with $U \in  C^{\infty}(D_0; \ccrmr_r)$ and $B \in  C^{\infty}(D_0; \ccrrn_r)$ for some open set $D_0$ containing $w_0$. We also have $U^T (w) U(w) \equiv \identity_r$, so that Theorem \ref{thmInverseRepresentation} shows that $\amp$ on $D_0$ may be expressed as
$$ \amp (w) = B^T(w) \left(B(w)B^T(w)\right)^{-1}U(w)^T \ldotp$$
All that is left to observe is that all the factors of the above expression are of class $C^{\infty}$. In particular, since the (square) matrix $B(w)B^T(w)$ is smooth and invertible at each point $w \in D_0$, its inverse is necessarily smooth as well.
\end{proof}
Finally, it is enough to observe that smoothness of $\bbQ$ immediately yields the same for $\bbP$, and thus we are done.

\end{document}